\documentclass[a4paper,12pt]{amsart}
\usepackage{amssymb,amsxtra,paralist,mathrsfs,upgreek,xcolor}
\usepackage[colorlinks=true,linkcolor=magenta,citecolor=cyan,urlcolor=olive]{hyperref}

\usepackage[left=3cm,right=3cm]{geometry}
\numberwithin{equation}{section}
\newtheorem{theoremA}{Theorem}

\newtheorem{theorem}{Theorem}

\newtheorem{lemma}[theorem]{Lemma}
\newtheorem{proposition}[theorem]{Proposition}

\theoremstyle{remark}
\newtheorem{remark}[theorem]{Remark}
\numberwithin{theorem}{section}

\newcommand{\bH}{\mathbb{H}}

\newcommand{\bR}{\mathbb{R}}

\newcommand{\Real}{\operatorname{Re}}
\newcommand{\Image}{\operatorname{Im}}

\newcommand{\bC}{\mathbb{C}}

\newcommand{\Id}{\operatorname{Id}}
\newcommand{\e}{\mathrm{e}}
\newcommand{\spec}{\operatorname{spec}}
\newcommand{\sF}{\mathscr{F}}
\newcommand{\bc}{\mathbf{c}}
\newcommand{\0}{\mathbf{0}}
\newcommand{\arcosh}{\operatorname{arcosh}}
\newcommand{\bJ}{\mathbf{J}}
\newcommand{\1}{\mathbf{1}}
\newcommand{\SO}{\mathrm{SO}}
\newcommand{\bG}{\mathbb{G}}
\newcommand{\bK}{\mathbb{K}}
\newcommand{\bA}{\mathbb{A}}
\newcommand{\sK}{\mathscr{K}}
\newcommand{\sE}{\mathscr{E}}
\newcommand{\tkappa}{\widetilde{\kappa}}
\newcommand{\fk}{\mathfrak{k}}
\newcommand{\fa}{\mathfrak{a}}
\newcommand{\bI}{\mathbf{I}}
\newcommand{\bN}{\mathbb{N}}
\newcommand{\sS}{S}
\newcommand{\tk}{\widetilde{k}}
\newcommand{\bZ}{\mathbb{Z}}
\renewcommand{\O}{\mathrm{O}}
\renewcommand{\d}{\mathrm{d}}
\renewcommand{\i}{\mathrm{i}}

\title[Inequalities for radial solutions to the wave equation]{Inequalities for radial solutions to the wave equation on the hyperbolic space}

\author[P. Chen, X.T. Duong, Y.X. Wang and L.X. Yan]{Peng Chen,\ Xuan Thinh Duong,\ Yunxiang Wang\ and\ Lixin Yan}
\address{Peng Chen, Department of Mathematics, Sun Yat-Sen (Zhongshan) University, Guangzhou, 510275, P.R. China}
\email{achenpeng1981@163.com}
\address{Xuan Thinh Duong, Department of Mathematics, Macquarie University, NSW 2109, Australia}
\email{xuan.duong@mq.edu.au}
\address{Yunxiang Wang, Department of Mathematics, Sun Yat-Sen (Zhongshan) University, Guangzhou, 510275, P.R. China}
\email{wangyx677@mail.sysu.edu.cn}
\address{Lixin Yan, Department of Mathematics, Sun Yat-Sen (Zhongshan) University, Guangzhou, 510275, P.R. China}
\email{mcsylx@mail.sysu.edu.cn}
\date{\today}
\subjclass[2020]{43A85, 42B15, 35L05}

\keywords{Local smoothing estimate, wave equation, radial function, hyperbolic space.}

\allowdisplaybreaks
\begin{document}

\begin{abstract}
Let $\mathbb H^n$ be the $n$-dimensional hyperbolic space, and $\Delta_{\mathbb H^n}$ be the Laplace--Beltrami operator on $\mathbb H^n$. Let $u=u(t,\cdot)$ be the solution to the wave equation $\partial^2_{t}u-\big(\Delta_{\mathbb H^n}+\frac{1}{4}(n-1)^2\operatorname{Id}\big)u=0$ with initial conditions $u(0,x)=f(x)$ and $\partial_tu(0,x)=0$. In this article we show that for a radial function $f$, $2n/(n-1)\leq p<\infty$, $\gamma<1/p$ and $T\geq 0$, there exists a constant $C_p$ independent of $f$ and $T$ such that
$$
\left(\int_T^{T+1}\|u(t,\cdot)\|^p_p\mathrm d t\right)^{\frac1p}\leq C_{p,\gamma}\,\mathrm e^{(n-1)\left(\frac12-\frac1p\right)T}\left\|(-\Delta_{\mathbb H^n})^{\frac{n-1}2\left(\frac12-\frac1p\right)-\frac\gamma2}f\right\|_p.
$$
We also show that the exponential part $\mathrm e^{(n-1)(1/2-1/p)T}$ of the bound is optimal. 
\end{abstract}

\maketitle

\section{Introduction}\label{sec:1}
\setcounter{equation}{0}

Let $\bH^n$ be the $n$-dimensional hyperbolic space, and $\Delta_{\bH^n}$ be the Laplace--Beltrami operator on $\bH^n$. Recall that $\spec\Delta_{\bH^n}\subset(-\infty,-(n-1)^2/4]$ on $L^2(\bH^n)$. We consider the Cauchy problem for the wave equation
\begin{align}\label{eq:1}
\begin{cases} 
&\partial^2_tu (t,x)-\left(\Delta_{\bH^n}+\frac{(n-1)^2}4\Id\right)\!u(t,x)=0,\quad (t,x)\in\bR\times \bH^n,\smallskip\\
&u(0,x)=f(x),\smallskip\\
&\partial_t u(0,x)=0.
\end{cases}
\end{align}
In 2000 Ionescu \cite[Corollary 1]{Ion00} obtained the following sharp fixed-time $L^p$-estimates.
\begin{theoremA}\label{thmA:A}
Let $u=u(t,x)$ be the solution to \eqref{eq:1}. Suppose $t\in\bR\backslash\{0\}$, $1<p<\infty$ and
\begin{align}\label{e1.1}
\alpha(p,n)= (n-1)\left|\frac1p-\frac12\right|.
\end{align}
Then
\begin{align}\label{eq:2}
\|u(t,\cdot)\|_p\leq C_p\,\e^{(n-1)\left|\frac1p-\frac12\right||t|}\left\|(-\Delta_{\bH^n})^{ {\alpha(p,n)}/ 2}f\right\|_p.
\end{align}
Moreover, the exponential part of the bound $\e^{(n-1)\left|\frac1p-\frac12\right||t|}$ depending on $t$ in \eqref{eq:2} is best possible.
\end{theoremA}

The Euclidean counterpart to Theorem \ref{thmA:A} was derived by Miyachi \cite{Miy80} and Peral \cite{Per80}. Let $\Delta_{\bR^d}$ be the standard Laplacian on $\bR^d$. They showed that for $t\in\bR\backslash\{0\}$, $1<p<\infty$ and $\alpha(p,d)= (d-1)|1/p-1/2|$, the solution $v=v(t,x)$ to the wave equation
\begin{align*} 
\begin{cases} 
&\displaystyle {\partial^2_tv} (t,x)-\Delta_{\bR^d}v(t,x)=0,\quad (t,x)\in\bR\times \bR^d,\smallskip\\
&v(0,x)=g(x),\smallskip\\
&\displaystyle \partial_t v (0,x)=0 
\end{cases}
\end{align*}
satisfies
\begin{align}\label{eq:3}
\|v(t,\cdot)\|_{L^p(\bR^d)}\leq C_p\, (1+|t|)^{(d-1)\left|\frac1p-\frac12\right|}\|g\|_{L^p_{\alpha(p,d)}(\bR^d)},
\end{align}
where $L^p_\alpha(\bR^d)$ stands for the standard inhomogeneous $L^p$ Sobolev space on $\bR^d$. It was shown in \cite{MSS92, Sog91} that averaging in time yields a gain in regularity, if $p>2$ and $d\geq 2$; that is, there is $\gamma=\gamma(p,d)>0$ such that for $2<p<\infty$ and $\alpha(p,d)=(d-1)|1/p-1/2|$,
\begin{align}\label{eq:4}
\left(\int_1^2\|v(t,\cdot)\|^p_{L^p(\bR^d)}\d t\right)^{\frac1p}\leq C_p \|g\|_{L^p_{\alpha(p,d)-\gamma}(\bR^d)}.
\end{align} 
In \cite{Sog91}, Sogge conjectured that if $2d/(d-1)\leq p<\infty$, this local smoothing phenomenon should hold with $\gamma<1/p$. This conjecture has attracted considerable attentions over the decades; for an overview of the progress, see \cite{MSS92,Wol00,HNS11,GWZ20,GHLW26} and the references therein. In particular, M\"uller and Seeger \cite{MS95} showed that the conjecture is true for spherically symmetric initial data in 1995. Recently, Guth, Wang and Zhang \cite{GWZ20} resolved Sogge's conjecture for the $2$-dimensional case, while in higher dimensions $d>2$ the conjecture still remains widely open, with the best-known result due to Gan et al. \cite{GHLW26}.

Inspired by the result of M\"uller and Seeger \cite{MS95}, the purpose of this paper is to study the local smoothing phenomenon for wave equation with radial initial data on $\bH^n$. Our main result is the following.
\begin{theorem}\label{thm:1}
Let $u=u(t,x)$ be the solution to \eqref{eq:1}. Suppose $2n/(n-1)\leq p<\infty$ and $f$ is a radial function in $\bH^n$. Let $\gamma<1/p$. Then for $T\geq 0$
\begin{align}\label{eq:4.1}
\left(\int_T^{T+1}\|u(t,\cdot)\|^p_p\d t\right)^{\frac1p}\leq C_{p,\gamma}\, \e^{(n-1)\left(\frac12-\frac1p\right)T}\left\|(-\Delta_{\bH^n})^{(\alpha(p,n)-\gamma)/2}f\right\|_p 
\end{align}
with $\alpha(p,n)$ as in \eqref{e1.1}.
Moreover, the exponential part $\e^{(n-1)(1/2-1/p)T}$ of the bound is optimal.
\end{theorem}

The paper is organized as follows. In Section \ref{sec:2} we review basic facts about $\bH^n$, including the spherical Fourier transform and spectral multipliers. With these preparations we are able to derive the kernel estimates of the wave propagator and therefore show our main result Theorem \ref{thm:1} in Section \ref{sec:3}. Finally, in Section \ref{sec:A}, we prove endpoint cases of radial local smoothing estimates for $2\leq p<2n/(n-1)$ (cf. Proposition \ref{prop:3.1}) for the case $n=3$ by using the explicit kernel expression of the wave propagator.

\bigskip

%$$$$$$$$$$$$$$$$$$$$$$$$$$$$$$$$$$$$$$$$$$$$$$$$$$$$$$$$$$$$$$$$$$$$$$$$$$

\section{Preliminaries}\label{sec:2}
\setcounter{equation}{0}

In this section we recall some basic facts about the $n$-dimensional hyperbolic space $\bH^n$ as well as its Fourier theory (see for examples, \cite{Hel00, IS09, Pet16}).

\subsection{The hyperbolic space}

We define the $n$-dimensional hyperblic space $\bH^n$ as the hyperboloid
\begin{align*}
\bH^n=\{x=(x_0,x')^\top\in \bR^{n+1}:[x,x]=1,\,x_0>0\}
\end{align*}
and equip this space with the Riemannian metric induced by the Minkowski metric $[\cdot,\cdot]$ on $\bR^{n+1}$ given by
\begin{align}\label{eq:2-0}
[x,y]=x_0y_0-\sum_{j=1}^nx_jy_j,\quad x=(x_0,\dots, x_n)^\top,\ y= (y_0,\dots,y_n)^\top\in {\bR^{n+1}}.
\end{align}
This Riemannian metric induces in turn a geodesic distance $d(\cdot,\cdot)$ given by
\begin{align}\label{eq:2-1}
d(x,y)=\arcosh([x,y]),
\end{align}
as well as a measure $\d x$. We will further distinguish the point $\0=(1,0,\dots,0)^\top$.

$\bH^n$ can also be realized as the symmetric space $\bG/\bK$, where $\bG=\SO_0(1,n)$ is the group of isometries of the Minkowski space $(\bR^{n+1},[\cdot,\cdot])$ leaving $\bH^n$ invariant, and
\begin{align*}
\bK=\left\{\left(\begin{array}{cc}
1 & 0 \\
0 & T
\end{array}\right):T\in\SO(n)\right\}.
\end{align*}
The space $\bG$ admits the Cartan decomposition $\bG=\bK\overline{\bA_+}\bK$, where
\begin{align*}
\overline{\bA_+}=\left\{\fa_r:=\left(\begin{array}{ccc}
\cosh r & \sinh r & 0 \\
\sinh r & \cosh r & 0 \\
0 & 0 & \Id_{n-1} 
\end{array}\right):r\geq 0\right\}
\end{align*}
We identify $\bG/\bK\simeq \bK\overline{\bA_+}$ with $\bH^n$ by the map $\fk\fa_r\mapsto\fk\fa_r\cdot\0$. Note that according to \eqref{eq:2-0} and \eqref{eq:2-1}, the geodesic distance $d(\cdot,\cdot)$ is invariant under left translations by $\bK$, i.e., for any $x,y\in\bH^n$ and $\fk\in\bK$,
\begin{align}\label{eq:2-2}
d(\fk\cdot x,\fk\cdot y)=d(x,y).
\end{align}
In particular, it follows from \eqref{eq:2-0} and \eqref{eq:2-1} that for $x=\fk\fa_r\in\bH^n$,
\begin{align}\label{eq:2-3}
d(x,\0)=d(\fa_r,\fk^{-1}\0)=d(\fa_r,\0)=r.
\end{align}
Furthermore, under these coordinates
\begin{align}\label{eq:7.1}
\int_{\bH^n}f(x)\,\d x=\frac{2\uppi^{\frac n2}}{\Gamma\!\left(\frac n2\right)}\int_0^\infty\int_{\bK}f(\fk\fa_r)\,(\sinh r)^{n-1}\d \fk\,\d r,
\end{align}
where $\d\fk$ is the normalized Haar measure on $\bK$.

\subsection{Spherical Fourier transform on \texorpdfstring{$\bH^n$}{Hn}}

Radial functions on $\bH^n$ depend on $R(\cdot):=d(\cdot,\0)$ only. In other words, for such a function $f$, there is another function $h$ defined on $[0,\infty)$ such that $f(x)=h(R(x))$, and it follows from \eqref{eq:2-3} and \eqref{eq:7.1} that
\begin{align}\label{eq:2-4}
\int_{\bH^n}f(x)\,\d x=\frac{2\uppi^{\frac n2}}{\Gamma\!\left(\frac n2\right)}\int_0^\infty h(r)\,(\sinh r)^{n-1}\d r.
\end{align}

The spherical Fourier transform of the above function $f$ can be written as
\begin{align}\label{eq:5}
\sF(f)(s)=\frac{2\uppi^{\frac n2}}{\Gamma\!\left(\frac n2\right)}\int_0^\infty h(r)\, \varphi_s(r)\,(\sinh r)^{n-1}\d r
\end{align}
and the inverse Fourier transform is given by
\begin{align}\label{eq:2-5}
f(x)=\frac{2^{n-3}\Gamma\!\left(\frac n2\right)}{\uppi^{\frac n2+1}}\int_0^{\infty} \sF(f)(s)\,\varphi_s(R(x))\, |\bc(s)|^{-2}\,\d s.
\end{align}
Here $\varphi_s(r)$ is the spherical function on $\bH^n$ given by
\begin{align}\label{eq:8}
\varphi_s(r)=\frac{2^{\frac{n-3}2}\Gamma\!\left(\frac n2\right)}{\sqrt{\uppi}\,\Gamma\!\left(\frac{n-1}2\right)}(\sinh r)^{2-n}\int_{-r}^r\e^{\i s\tau}(\cosh r-\cosh\tau)^{\frac{n-3}2}\d \tau,
\end{align}
see \cite[(3.1.4)]{Bra94}, and $\bc$ is the Harish-Chandra function given by
\begin{align}\label{eq:7}
\bc(s)=\frac{2^{n-2}\Gamma\!\left(\frac n2\right)}{\sqrt{\uppi}}\frac{\Gamma(\i s)}{\Gamma\!\left(\frac{n-1}2+\i s\right)}.
\end{align}

Let
\begin{align*}
\sS^\sigma=\left\{m\in C^\infty(\bR):\sup_{s\in\bR}(1+|s|)^{-\sigma+\ell}|m^{(\ell)}(s)|\leq C_\ell\text{ for }\ell=0,1,\dots\right\}.
\end{align*}
be the symbol class of order $\sigma$. We have the following asymptotic expansions of the Harish-Chandra $\bc$-function and the spherical function $\varphi_s(r)$.
\begin{lemma}\label{lem:3.0}
\begin{asparaenum}[\rm (i)]
\item\label{lem:3.0(ii)} The functions $(s\,\bc(s))^{-1},\ (s\,\bc(-s))^{-1}\in S^{(n-3)/2}$, while $s\,\bc(s),\ s\,\bc(-s)\in S^{-(n-3)/2}$.
\item\label{lem:3.0(iii)} $|\bc(\cdot)|^{-2}=(\bc(s)\,\bc(-s))^{-1}\in\sS^{n-1}$. In particular, for $s\geq 2$,
\begin{align*}
|\bc(s)|^{-2}=\left(\frac{2^{n-2}\Gamma\!\left(\frac n2\right)}{\sqrt{\uppi}}\right)^{-2}\left(s^{n-1}+a(s)\right),
\end{align*}
where $a\in \sS^{n-2}$.
\end{asparaenum}
\end{lemma}
\begin{proof}
\eqref{lem:3.0(ii)} follows from \eqref{eq:7} and Stirling's formula \cite[p. 47, (4)]{EMOT53}.\smallskip

The first assertion of \eqref{lem:3.0(iii)} follows from \eqref{eq:7} and the property that $\overline{\Gamma(z)}=\Gamma(\overline{z})$, while the last assertion of \eqref{lem:3.0(iii)} is a straightforward consequence of Stirling's formula \cite[p. 47, (4)]{EMOT53}.
\end{proof}
\begin{lemma}\label{lem:3.1}
\begin{asparaenum}[\rm (i)]
\item\label{lem:3.1(i)} For $0<r\leq 1$,
\begin{align*}
\varphi_s(r)=(\sinh r)^{-\frac{n-1}2}\left(\e^{\i sr}a(s,r)+\e^{-i sr}a(-s,r)\right),
\end{align*}
where $a(\cdot,r)\in\sS^{-(n-1)/2}$ uniformly in $0<r\leq 1$.
\item\label{lem:3.1(ii)} For $r\geq 1/2$,
\begin{align*}
\varphi_s(r)=(\sinh r)^{-\frac{n-1}2}\left(\e^{\i sr}\bc(s)\,b(s,r)+\e^{-\i sr}\bc(-s)\,b(-s,r)\right)\!,
\end{align*}
where $b(\cdot,r)\in \sS^0$ uniformly in $r\geq 1/2$. In particular, if $s\geq 2$, then
\begin{align*}
\varphi_s(r)&=\frac{2^{\frac{n-1}2}}{\sqrt{\uppi}}\Gamma\!\left(\frac n2\right)(\sinh r)^{-\frac{n-1}2}s^{-\frac{n-1}2}\cos\!\left(sr-\frac{(n-1)\uppi}4\right)\\
&\qquad+(\sinh r)^{-\frac{n-1}2}\big(\e^{\i sr}a(s,r)+\e^{-\i sr}a(-s,r)\big),
\end{align*}
where $a(\cdot,r)\in\sS^{-(n+1)/2}$ uniformly in $r\geq 1/2$.
\end{asparaenum}
\end{lemma}
\begin{proof}
To show \eqref{lem:3.1(i)} we shall use the following asymptotic expansion \cite[Theorem2.1]{ST78} for $0<r\leq 1$:
\begin{align}\label{eq:7.0}
\varphi_s(r)&=\frac{\Gamma\!\left(\frac n2\right)}{\Gamma\!\left(\frac{n-1}2\right)}(\sinh r)^{-\frac{n-1}2}\sum_{\ell=0}^\infty 2^{\frac{n-2}2+\ell}\Gamma\!\left(\frac{n-1}2+\ell\right)c_\ell(r)\\
&\hphantom{=\frac{\Gamma(\frac n2)}{\Gamma(\frac{n-1}2)}(\sinh r)^{-\frac{n-1}2}\sum_{\ell=0}^\infty}\ \ \ \ \times r^{\frac{n-1}2+2\ell}(sr)^{-\frac{n-2}2-\ell}J_{\frac{n-2}2+\ell}(sr),\notag
\end{align}
where $J_\mu$ is the Bessel function of order $\mu$, and the coefficients $c_\ell$ satisfy
\begin{align*}
c_0\equiv 1,\quad |c_\ell(t)|\leq CR^{-\ell}\ (\ell\in\bN^*)
\end{align*}
for some constant $R>1$. It's well-known that the Bessel function satisfies the following asymptotic expansion:
\begin{align*}
\rho^{-\mu}J_\mu(\rho)=\e^{\i\rho}\psi_{-\mu-1/2}(\rho)+\e^{-\i\rho}\psi_{-\mu-1/2}(-\rho)\quad\text{for }\rho>0,
\end{align*}
where $\psi_{-\mu-1/2}\in\sS^{-\mu-1/2}$. Since for $\ell\in\bN$,
\begin{align*}
r^{\frac{n-1}2+2\ell}\psi_{-\frac{n-1}2-\ell}(r\cdot)\in\sS^{-\frac{n-1}2-\ell}\quad\text{uniformly in }0<r\leq 1,
\end{align*}
in view of \eqref{eq:7.0}, for $r$ as above the spherical function $\varphi_s(r)$ satisfies the desired asymptotic expansion.\smallskip

The first assertion in \eqref{lem:3.1(ii)} follows readily from the Harish-Chandra expansion \cite[(8) and (10)]{APV12}:
\begin{align}\label{eq:7.01}
\varphi_s(r)=(2\sinh r)^{-\frac{n-1}2}\left(\e^{\i sr}\bc(s)\sum_{\ell=0}^\infty\frac{\gamma_\ell(s)}{\e^{2\ell r}}+\e^{-\i sr}\bc(-s)\sum_{\ell=0}^\infty\frac{\gamma_\ell(-s)}{\e^{2\ell r}}\right)\!,
\end{align}
where the coefficients $\gamma_\ell$ satisfy
\begin{align}\label{eq:3-3.3}
\gamma_0\equiv 1,\quad|\partial^j_s\gamma_\ell(s)|\leq C_j\ell^\nu (1+|s|)^{-j-1}\ (\ell\in\bN^*,\ j\in\bN,\ s\in\bR^*)
\end{align}
for some constant $\nu>0$. The last assertion of \eqref{lem:3.1(ii)} is indeed \cite[Proposition 3.1]{WY25}.
\end{proof}

\subsection{Spectral multipliers on \texorpdfstring{$\bH^n$}{Hn}}

In the following, we denote
\begin{align*}
D=\sqrt{-\Delta_{\bH^n}-\frac{(n-1)^2}4\Id},
\end{align*}
where $\Delta_{\bH^n}$ is the Laplace--Beltrami operator on $\bH^n$ with $\spec\Delta_{\bH^n}\subset (-\infty,-(n-1)^2/4]$ on $L^2(\bH^n)$. For any bounded Borel function $m:[0,\infty)\to\bC$, we define the radial Fourier multiplier $m(D)$ by the formula
\begin{eqnarray}\label{eq:9}
m(D)(f)(x)=f*k_m(x):=\int_{\bH^n}f(y)\,k_m(d(x,y))\,\d y,
\end{eqnarray}
where $k_m$ satisfies $m=\sF(k_m)$, or equivalently
\begin{align}\label{eq:9.0}
k_m(r)=\frac{2^{n-3}\Gamma\!\left(\frac n2\right)}{\uppi^{\frac n2+1}}\int_0^{\infty} m(s)\,\varphi_s(r)\, |\bc(s)|^{-2}\,\d s
\end{align}
by the Fourier inversion formula \eqref{eq:2-5}. In addition, it follows from the spectral theorem that
\begin{align}\label{eq:9.1}
\|f*k_m\|_2=\|m(D)(f)\|_2\leq \|m\|_{L^\infty([0,\infty))}\|f\|_2=\|\sF(k_m)\|_{L^\infty([0,\infty))}\|f\|_2.
\end{align}

Particularly, for any radial function $f$, in view of \cite[p. 10]{GL23} and \eqref{eq:5}, we have
\begin{align}
m(D)(f)(x)&=C\int_0^\infty \sF(f)(s)\,m(s)\,\varphi_s(R(x))\,|\bc(s)|^{-2}\,\d s\notag\\
&=C\int_0^\infty\!\left(\int_0^\infty h(r)\,\varphi_s(r)\,(\sinh r)^{n-1}\,\d r\right)m(s)\,\varphi_s(R(x))\,|\bc(s)|^{-2}\,\d s\notag\\
&=C\int_0^\infty h(r)\,\kappa_m(R(x),r)\,\d r,\label{eq:10}
\end{align}
where $R(x)=d(x,\0)$, $h$ is the function defined on $[0,\infty)$ such that $f(x)=h(R(x))$, and
\begin{align}\label{eq:11}
\kappa_m(R,r)&=(\sinh r)^{n-1}\int_0^\infty\varphi_s(R)\,\varphi_s(r)\,m(s)\,|\bc(s)|^{-2}\,\d s.
\end{align}
On the other hand, by \eqref{eq:2-2} and \eqref{eq:2-3}, we write \eqref{eq:9} in the Cartan coordinates \eqref{eq:7.1} to obtain
\begin{align}\label{eq:11.1}
m(D)(f)(x)&=C\int_0^\infty h(r)\,\left(\int_{\bK}k_m(d(x,\fk\fa_r))\,\d\fk\right)(\sinh r)^{n-1}\d r\notag\\
&=C\int_0^\infty h(r)\,\left(\int_{\bK}k_m(d(\fa_{R(x)},\fk\fa_r))\,\d\fk\right)(\sinh r)^{n-1}\d r.
\end{align}
Comparing \eqref{eq:10} with \eqref{eq:11.1}, we have
\begin{align}\label{eq:11.2}
\kappa_m(R,r)=C\,(\sinh r)^{n-1}\int_{\bK}k_m(d(\fa_R,\fk\fa_r))\,\d\fk
\end{align}
for a.e. $R,r>0$.

We finish this section by recalling the following multiplier theorem on $\bH^n$ was obtained by Clerc and Stein \cite{CS74}, see also \cite[Theorem A]{Tay89}.
\begin{proposition}[{\cite[Theorem 3]{CS74}}]\label{prop2.3}
	Let $1<p<\infty$. Suppose $m$ is an even bounded Borel function on $[0, \infty)$, analytic inside the tube $\mathcal T_p:=\{z\in \bC: |\Image z|<(n-1)|1/2-1/p|\}$, smooth on $\overline{\mathcal{T}}_p$, satisfying the symbol estimate
	\begin{align*}
		\left|m^{(\ell)}(\zeta)\right|\lesssim_\ell (1+|\zeta|)^{-\ell},\qquad \zeta\in\overline{\mathcal{T}}_p\text{ and } \ell=0,1,\dots.
	\end{align*}
	Then $m(D)$ is bounded on $L^p(\bH^n)$.
\end{proposition}

\bigskip

%$$$$$$$$$$$$$$$$$$$$$$$$$$$$$$$$$$$$$$$$$$$$$$$$$$$$$$$$$$$$$$$$$$$$$$$$$$

\section{Proof of Theorem \ref{thm:1}}\label{sec:3}
\setcounter{equation}{0}

To show Theorem \ref{thm:1}, we need the following proposition, which plays an important role in the proof of Theorem \ref{thm:1}. 
\begin{proposition}\label{prop:3.1}
 Suppose $2\leq p<2n/(n-1)$ and $f$ is a radial function in $\bH^n$, and $T>0$. Then for any $\delta>0$
\begin{align}\label{eq:3-1}
\left(\int_T^{T+1}\left\|(-\Delta_{\bH^n})^{-\frac\delta2}\cos(tD)(f)\right\|^p_p\d t\right)^{\frac1p}\lesssim_{p,\delta} \e^{(n-1)\big(\frac12-\frac1p\big)T}\|f\|_p.
\end{align}
\end{proposition}

\begin{proof}
By Proposition~\ref{prop2.3}, it suffices to show that for every $0<\delta<1/2$,
\begin{align}\label{eq:3-1r}
\left(\int_T^{T+1}\left\|\left(-\Delta_{\bH^n}+\frac{1}{10}\Id\right)^{-\frac\delta2}\cos(tD)(f)\right\|^p_p\d t\right)^{\frac1p}\lesssim_{p,\delta} \e^{(n-1)\big(\frac12-\frac1p\big)T}\|f\|_p.
\end{align}

We define
\begin{align*}
k_{\delta,t}(r):=\sF^{-1}\!\left[\left((\cdot)^2+\frac{(n-1)^2}4+\frac1{10}\right)^{-\frac\delta2}\cos(t\cdot)\right](r)
\end{align*}
such that
\begin{align*}
\left(-\Delta_{\bH^n}+\frac{1}{10} \Id \right)^{-\frac\delta2}\cos(tD)(f)(x)=\int_{\bH^n}f(y)\,k_{\delta,t}(d(x,y))\,\d y,
\end{align*}
and decompose the kernel $k_{\delta,t}(d(x,y))$ into $K_{\delta,t}(x,y)+E_{\delta,t}(x,y)$, where
\begin{align*}
K_{\delta,t}(x,y)=\1_{\{y\in\bH^n:|R(x)-R(y)|\leq t+1\}}(y)\,k_{\delta,t}(d(x,y)).
\end{align*}
We further define the operator
\begin{align*}
\sK_{\delta,t}(f)(x):=\int_{\bH^n}f(y)\,K_{\delta,t}(x,y)\,\d y,\quad\sE_{\delta,t}(f)(x):=\int_{\bH^n}f(y)\,E_{\delta,t}(x,y)\,\d y.
\end{align*}

We first deal with the operator $\sK_{\delta,t}$. Since $f$ is radial, we write $|f(x)|=h(R(x))$ for some function $h$ defined on $[0,\infty)$. We also set $f_0(r):=(\sinh r)^{\frac{n-1}p}h(r)$ such that
\begin{align}
\|f_0\|_{L^p(\bR_+)}=C\|f\|_p.
\end{align}
By \eqref{eq:2-2} and \eqref{eq:2-3}, we write in the Cartan coordinates
\begin{align*}
\sK_{\delta,t}(f)(x)&=C\int_0^\infty f(r)\,\1_{\{r\in\bR_+:|R(x)-r|\leq t+1\}}(r)\,(\sinh r)^{n-1}\\
&\hphantom{=C\int_0^\infty}\ \ \times\left(\int_{\bK}k_{\delta,t}(d(\fa_{R(x)},\fk\fa_r))\,\d\fk\right)\d r.
\end{align*}
Therefore, in view of \eqref{eq:2-4}, \eqref{eq:11} and \eqref{eq:11.2},
\begin{align}\label{eq:3-2}
\|\sK_{\delta,t}(f)\|_p^p\lesssim\int_0^\infty\left(\int_0^\infty f_0(r)\,\kappa_{\delta,t}(R,r)\,\d r\right)^p\d R,
\end{align}
where
\begin{align*}
\kappa_{\delta,t}(R,r)&=\1_{\{(R,r)\in(\bR_+)^2:|R-r|\leq t+1\}}(R,r)\,(\sinh r)^{(n-1)\left(1-\frac1p\right)}(\sinh R)^{\frac{n-1}p}\\
&\qquad\times\left|\vphantom{\int_0^\infty \left(s^2+\frac{(n-1)^2}4\right)^{-\frac\delta2}\cos(st)\,\varphi_s(R)\,\varphi_s(r)\,|\bc(s)|^{-2}\d s}\right.\underbrace{\int_0^\infty \left(s^2+\frac{(n-1)^2}4+\frac1{10}\right)^{-\frac\delta2}\cos(st)\,\varphi_s(R)\,\varphi_s(r)\,|\bc(s)|^{-2}\d s}_{\bI(R,r,t)}\left.\vphantom{\int_0^\infty \left(s^2+\frac{(n-1)^2}4+\frac1{10}\right)^{-\frac\delta2}\cos(st)\,\varphi_s(R)\,\varphi_s(r)\,|\bc(s)|^{-2}\d s}\right|.
\end{align*}
In view of Lemmas \ref{lem:3.0} and \ref{lem:3.1},
\begin{align*}
\bI(R,r,t)&=(\sinh r\cdot \sinh R)^{-\frac{n-1}2}\sum_{\iota_1,\iota_2\in\{-1,1\}}\int_{\bR}\e^{\i s(t+\iota_1R+\iota_2r)}a_{\delta,\iota_1,\iota_2}(s,R,r)\,\d s,
\end{align*}
where $a_{\delta,\pm1,\pm1}(\cdot,R,r)\in\sS^{-\delta}$ uniformly in $R,r>0$. By the classical estimates on Fourier transform of symbols \cite[p. 241]{Ste93},
\begin{align*}
|\bI(R,r,t)|\lesssim (\sinh r)^{-\frac{n-1}2}(\sinh R)^{-\frac{n-1}2}\sum_{\iota_1,\iota_2\in\{-1,1\}}A_\delta(|t+\iota_1R+\iota_2r|)
\end{align*}
for
\begin{align*}
A_\delta(\tau)=\tau^{\delta-1}(1+|\tau|)^{-100},
\end{align*}
whence
\begin{align*}
\kappa_{\delta,t}(R,r)&\lesssim\1_{\{(R,r)\in(\bR_+)^2:|R-r|\leq t+1\}}(R,r)\,\left(\frac{\sinh r}{\sinh R}\right)^{(n-1)\left(\frac12-\frac1p\right)}\\
&\qquad\times\sum_{\iota_1,\iota_2\in\{-1,1\}}A_\delta(|t+\iota_1R+\iota_2r|),
\end{align*}
Now we integrate on both sides of \eqref{eq:3-2} with respect to $t$ over $[T,T+1]$ to get
\begin{align}
&\int_T^{T+1}\|\sK_{\delta,t}(f)\|_p^p\d t\notag\\
&\quad\lesssim \e^{(n-1)\left(\frac p2-1\right)T}\int_0^1R^{(n-1)\left(1-\frac p2\right)}\int_T^{T+1}\left(\int_0^\infty f_0(r)\,\tkappa_{\delta,R}(t,r)\,\d r\right)^p\d t\,\d R\label{eq:19}\\
&\quad\qquad+\int_T^{T+1}\int_1^\infty\left(\int_0^\infty f_0(r)\,\kappa_{\delta,t}(R,r)\,\d r\right)^p\,\d R\,\d t\notag
\end{align}
provided $p\geq 2$, where
\begin{align*}
\tkappa_{\delta,R}(t,r):=\1_{[0,R+t+1]}(r)\sum_{\iota_1,\iota_2\in\{-1,1\}}A_\delta(|t+\iota_1R+\iota_2r|).
\end{align*}

On the one hand, for $R\in(0,1)$ and $\delta>0$, 
\begin{align*}
\sup_{t\in[T,T+1]}\int_0^\infty \tkappa_{\delta,R}(t,r)\,\d r+\sup_{r\in(0,\infty)}\int_T^{T+1}\tkappa_{\delta,R}(t,r)\,\d t\leq C_\delta.
\end{align*}
On the other hand, for $\delta$ as above, $t\in[T,T+1]$ and $p\geq 2$,
\begin{align*}
\sup_{R\in[1,\infty)}\int_0^\infty \kappa_{\delta,t}(R,r)\,\d r+\sup_{r\in(0,\infty)}\int_1^\infty \kappa_{\delta,t}(R,r)\,\d R\lesssim_\delta \e^{(n-1)\left(\frac12-\frac1p\right)T}.
\end{align*}
By Schur's test, we conclude that for all $\delta,p$ as above and $1\leq q\leq \infty$,
\begin{align*}
\sup_{R\in(0,1)}\int_T^{T+1}\left(\int_0^\infty f_0(r)\,\tkappa_{\delta,R}(t,r)\,\d r\right)^q\d t\lesssim_\delta \|f_0\|_{L^q(\bR_+)}^q
\end{align*}
and
\begin{align*}
\sup_{t\in[T,T+1]}\int_1^\infty\left(\int_0^\infty f_0(r)\,\kappa_{\delta,t}(R,r)\,\d r\right)^q\,\d R\lesssim_\delta \e^{q(n-1)\left(\frac12-\frac1p\right)T}\|f_0\|_{L^q(\bR_+)}^q.
\end{align*}
Therefore, according to \eqref{eq:19},
\begin{align}\label{eq:3-13}
\left(\int_T^{T+1}\|\sK_{\delta,t}(f)\|_p^p\d t\right)^{\frac1p}\lesssim_{p,\delta} \e^{(n-1)\left(\frac12-\frac1p\right)T}\|f_0\|_{L^p(\bR_+)}\sim\e^{(n-1)\left(\frac12-\frac1p\right)T} \|f\|_p
\end{align}
provided $2\leq p<2n/(n-1)$ and $\delta>0$.\medskip

It remains to handle the error term $\sE_{\delta,t}$. Indeed, on the one hand, note that the triangle inequality implies that $E_{\delta,t}$ is supported in $\{(x,y):d(x,y)>t+1\}$. Hence by the argument proving \cite[(3.9)]{Ion00},
\begin{align*}
|E_{\delta,t}(x,y)|\lesssim e_{\delta,t}(d(x,y))
\end{align*}
for
\begin{align*}
e_{\delta,t}(r)=\1_{(t+1,\infty)}(r)\,\e^{\frac{n-1}2t}\e^{-(n-1)r}(r-t)^{-2},
\end{align*}
whence
\begin{align*}
\|\sE_{\delta,t}(f)(x)\|_{L^\infty_{x,t}(\bH^n\times[T,T+1])}\leq \|e_{\delta,t}\|_1\|f\|_\infty\lesssim \e^{\frac{n-1}2T}\|f\|_\infty
\end{align*}
in view of Schur's test and \eqref{eq:2-4}. On the other hand, according to \eqref{eq:9.1} and \eqref{eq:3-13} with $p=2$,
\begin{align*}
&\|\sE_{\delta,t}(f)(x)\|_{L^2_{x,t}(\bH^n\times[T,T+1])}\\
&\ \ \leq \sup_{t\in[T,T+1]}\left\|\left(-\Delta_{\bH^n}+\frac{1}{10}\Id\right)^{-\frac\delta2}\cos(tD)(f)\right\|_2+\left(\int_T^{T+1}\|\sK_{\delta,t}(f)\|_2^2\d t\right)^{\frac12}\lesssim \|f\|_2.
\end{align*}
Therefore, by interpolation, for all $2\leq p\leq \infty$,
\begin{align*}
\left(\int_T^{T+1}\|\sE_{\delta,t}(f)\|^p_p\d t\right)^{\frac1p}\lesssim \e^{(n-1)\left(\frac12-\frac1p\right)T}\|f\|_p.
\end{align*}
This, combined with \eqref{eq:3-13}, yields the desired estimate \eqref{eq:3-1r}. The proof of Proposition~\ref{prop:3.1} is complete. 
\end{proof}

Now we start to prove Theorem \ref{thm:1}.
\begin{proof}[Proof of Theorem \ref{thm:1}] 
Recall that the solution to \eqref{eq:1} can be written as $u(t,x)=\cos(tD)(f)(x)$. We shall firstly use Theorem \ref{thmA:A} and Proposition \ref{prop:3.1} to establish the estimate
\begin{align}\label{eq:21}
\int_T^{T+1}\left\|(-\Delta_{\bH^n})^{-\frac\alpha2}\cos(tD)(f)\right\|_p^p\d t\lesssim_{\alpha,p}\e^{p(n-1)\big(\frac12-\frac1p\big)T}\|f\|_p^p
\end{align}
for $p\to \infty$ and $p\to (2n/(n-1))^-$, respectively, and then use interpolation. In the sequel we fix $\eta\in\bR$.

On the one hand, for $10000\leq p_1<\infty$ we have
\begin{align*}
&\int_T^{T+1}\left\|(-\Delta_{\bH^n})^{-\frac{n-1}2\left(\frac12-\frac1{p_1}\right)-\frac{\i\eta}2}\cos(tD)(f)\right\|_{p_1}^{p_1}\d t\\
&\qquad\lesssim_{p_1} \e^{(n-1)\big(\frac12-\frac1{p_1}\big)T}\left\|(-\Delta_{\bH^n})^{-\frac{\i\eta}2}f\right\|_{p_1}^{p_1}\leq C_\eta \e^{p_1(n-1)\big(\frac12-\frac1{p_1}\big)T}\|f\|_{p_1}^{p_1}
\end{align*}
with $C_\eta\lesssim\e^{c\eta^2}$, where the first inequality follows from Theorem \ref{thmA:A} and the last inequality follows from Proposition ~\ref{prop2.3}. Hence \eqref{eq:21} holds with $p=p_1$, $\Real\alpha=(n-1)(1/2-1/p_1)$ and the implicit constant of at most $\e^{c(\Image\alpha)^2}$ growth on $\Image\alpha$.

On the other hand, for $2\leq p_2<2n/(n-1)$ and $\delta>0$ we have
\begin{align*}
&\int_T^{T+1}\left\|(-\Delta_{\bH^n})^{-\frac\delta2-\frac{\i\eta}2}\cos(tD)(f)\right\|_{p_2}^{p_2}\d t\\
&\ \ \lesssim_{p_2,\delta} \e^{(n-1)\big(\frac12-\frac1{p_2}\big)T}\left\|(-\Delta_{\bH^n})^{-\frac{\i\eta}2}f\right\|_{p_2}^{p_2}\leq C_\eta\,\e^{p_2(n-1)\big(\frac12-\frac1{p_2}\big)T}\|f\|_{p_2}^{p_2}
\end{align*}
with $C_\eta\lesssim\e^{c\eta^2}$, where the first inequality follows from Proposition \ref{prop:3.1} and the last inequality follows from Proposition ~\ref{prop2.3}. Hence \eqref{eq:21} holds with $p=p_2$, $\Real\alpha=\delta$ and the implicit constant of at most $\e^{c(\Image\alpha)^2}$ growth on $\Image\alpha$.

Using the interpolation of analytic family of operators (see e.g. \cite{Sag69}), we obtain that \eqref{eq:21} holds for $p_2\leq p\leq p_1$ and
\begin{align*}
\alpha=(n-1)\left(\frac12-\frac1{p_1}\right)\frac{\frac1{p_2}-\frac1p}{\frac1{p_2}-\frac1{p_1}}+\delta\frac{\frac1{p_1}-\frac1p}{\frac1{p_1}-\frac1{p_2}}=(n-1)\left(\frac12-\frac1p\right)-\gamma,
\end{align*}
where
\begin{align*}
\gamma=(n-1)\left(\frac1{p_1}-\frac1p\right)+(n-1)\left(\frac12-\frac1{p_1}\right)\frac{\frac1p-\frac1{p_1}}{\frac1{p_2}-\frac1{p_1}}-\delta\frac{\frac1{p_1}-\frac1p}{\frac1{p_1}-\frac1{p_2}}<\frac1p.
\end{align*}
Consequently, \eqref{eq:21} holds with $2n/(n-1)\leq p<\infty$ and $\alpha=(n-1)(1/2-1/p)-\gamma$ with $\gamma<1/p$, as desired. This completes the proof of \eqref{eq:4.1}.

Now we show the optimality of the exponential part $\e^{(n-1)(1/p-1/2)T}$ of the bound in \eqref{eq:4.1} of Theorem~\ref{thm:1}. In particular, we shall show a somewhat stronger result as follows.
\begin{proposition}\label{prop:6}
For all $2\leq p<\infty$, $T\geq 0$ and $\alpha\in\bR$, if
\begin{align}\label{eq:16.1}
\left(\int_T^{T+1}\|\cos(tD)(f)\|^p_p\d t\right)^{\frac1p}\leq C_p(T)\,\|(-\Delta)^\alpha f\|_p
\end{align}
holds for all radial functions $f$ on $\bH^n$, then we must have
\begin{align*}
C_p(T)\gtrsim\e^{(n-1)\big(\frac12-\frac1p\big)T}.
\end{align*}
\end{proposition}
\begin{proof}
Instead of considering \eqref{eq:16.1}, we shall focus on another estimate
\begin{align}\label{eq:16.2}
\left(\int_T^{T+1}\left\|(-\Delta_{\bH^n})^{\frac18}\cos(tD)(f)\right\|^p_p\d t\right)^{\frac1p}\leq C_p(T)\,\|(-\Delta_{\bH^n})^{\alpha'} f\|_p\quad\text{for }\alpha'=\alpha+\frac18
\end{align}
equivalent to \eqref{eq:16.1}.

Let $c>0$ be a sufficiently small but fixed constant to be determined later, $h\in C^\infty_c(\bR_+)$ such that $0\leq h\leq 1$ and
\begin{align*}
h(r)=\begin{cases}
1 & \text{on }\big[\frac12+5c,\frac12+6c\big],\smallskip\\
0 & \text{outside }[\frac12+4c,\frac12+7c].
\end{cases}
\end{align*}
We set ad radial function $f_T(x):=h(R(x)-T)$ on $\bH^n$. Then it follows from the multiplier theorem (see e.g. \cite[Theorem 3]{CS74} or \cite[Theorem A]{Tay89}) that if we set $k:=\max\{0,\lceil\alpha'\rceil\}$, then
\begin{align}\label{eq:16.1.1}
\|(-\Delta_{\bH^n})^{\alpha'}f_T\|_p\lesssim \|(-\Delta_{\bH^n})^kf_T\|_p\lesssim \|h\|_{C^N}\|\mathbf{1}_{[T+\frac12+4c,T+\frac12+7c]}(R(\cdot))\|_p\lesssim \e^{\frac{n-1}pT},
\end{align}
where the second inequalities follows from the fact that
\begin{align*}
\Delta_{\bH^n} f_T(x)=\frac{\partial^2h(R(\cdot)-T)}{\partial r^2}(R(x))+(n-1)\coth(R(x))\frac{\partial h(R(\cdot)-T)}{\partial r}(R(x)),
\end{align*}
see e.g. \cite[(3.85)]{Gri09}.

Next we need to bound
\begin{align*}
\left(\int_T^{T+1}\left\|(-\Delta_{\bH^n})^{\frac18}\cos(tD)(f_T)\right\|^p_p\d t\right)^{\frac1p}
\end{align*}
from below. To do this, we denote by $k_t$ the convolution kernel of the operator $(-\Delta_{\bH^n})^{1/8}\cos(tD)$ and let $\eta$ be an even smooth function such that
\begin{align*}
\eta\equiv \begin{cases}
1 & \text{outside }[-3,3],\\
0 & \text{on }[-2,2].
\end{cases}
\end{align*}
Then according to \eqref{eq:9.0}, Lemmas \ref{lem:3.0}\eqref{lem:3.0(iii)} and \ref{lem:3.1}\eqref{lem:3.1(ii)}, for $t\in[T+1/2,T+1/2+c]$,
\begin{align}
k_t(x)&=\frac{2^{n-3}\Gamma\!\left(\frac n2\right)}{\uppi^{\frac n2+1}}\int_0^\infty \left(s^2+\frac{(n-1)^2}4\right)^{\frac18}\cos(ts)\,\varphi_s(R(x))\,|\bc(s)|^{-2}\d s\notag\\
&=C_n\,(\sinh R(x))^{-\frac{n-1}2}\!\left[\vphantom{\int_0^\infty \eta(s)\,s^{\frac n2-\frac14}\cos(ts)\cos\!\left(sR(x)-\frac{(n-1)\uppi}4\right)\d s}\right.\!\!\underbrace{\int_0^\infty \eta(s)\,s^{\frac n2-\frac14}\cos(ts)\cos\!\left(sR(x)-\frac{(n-1)\uppi}4\right)\d s}_{\tk_t(x)}\label{eq:16.3}\\
&\hphantom{C_n\e^{-\frac{n-1}2R(x)}}\ \ +\underbrace{\int_0^\infty \cos(ts)\left(\e^{\i sR(x)}a(s,R(x))+\e^{-\i sR(x)}a(-s,R(x))\right)\d s}_{e_t(x)}\left.\vphantom{\int_0^\infty \cos(ts)\left(\e^{\i sR(x)}a(s,R(x))+\e^{-\i sR(x)}a(-s,R(x))\right)\d s}\!\!\right]\!,\notag
\end{align}
where $a(\cdot,r)\in S^{n/2-5/4}$ uniformly in $r\geq 1/2$.

To handle $\tk_t$, for $r>0$ and $\gamma\notin2\bZ+1$, in view of \cite[Theorem 2.4.6]{Gra14}, we have
\begin{align}\label{e6.001}
\int_0^\infty\cos(sr)\,s^\gamma\d s=\frac{\Gamma\!\left(\frac{\gamma+1}2\right)}{\Gamma\!\left(-\frac\gamma2\right)}2^\gamma\sqrt{\uppi}\,r^{-1-\gamma}.
\end{align}
Taking derivatives with respect to $r$ on both sides of \eqref{e6.001} and then replacing $\gamma +1$ by $\gamma$, in view of \cite[p. 3, (6) and (7)]{EMOT53}, we obtain that for $\gamma\notin2\bZ$,
\begin{align*}
\int_0^\infty\sin(sr)\,s^\gamma\d s=-\frac{\Gamma\!\left(\frac{\gamma+1}2\right)}{\Gamma\!\left(-\frac\gamma2\right)}\cot\!\left(\frac{\gamma\uppi}2\right)2^\gamma\sqrt{\uppi}\,r^{-1-\gamma}.
\end{align*}
Since
\begin{align*}
&\cos\!\left(sr-\frac{(n-1)\uppi}4\right)\cos(ts)\\
&\quad=\frac12\!\left[\cos\!\left((r-t)s-\frac{(n-1)\uppi}4\right)+\cos\!\left((r+t)s-\frac{(n-1)\uppi}4\right)\right],
\end{align*}
we have
\begin{align*}
\int_0^\infty s^{\frac n2-\frac14}\cos(ts)\cos\!\left(sr-\frac{(n-1)\uppi}4\right)\d s=C\,(r-t)^{-\frac n2-\frac34}+\O(1)
\end{align*}
for $t<r\leq t+10c$ with $c$ small enough. In addition,
\begin{align*}
\int_0^\infty (1-\eta(s))\,s^{\frac n2-\frac14}\cos(ts)\cos\!\left(sr-\frac{(n-1)\uppi}4\right)\d s=\O(1).
\end{align*}
Hence
\begin{align}\label{eq:16.4}
\tk_t(x)=C\,(R(x)-t)^{-\frac n2-\frac34}+\O(1)
\end{align}
for $t\in[T+1/2,T+1/2+c]$ and $t<R(x)\leq t+10c$. On the other hand, by the classical estimates on Fourier transform of symbols \cite[p. 241]{Ste93}, for $c$ small enough and $t,x$ as above we have
\begin{align}\label{eq:16.5}
|e_t(x)|\lesssim (R(x)-t)^{\frac14-\frac n2}.
\end{align}

It follows from \eqref{eq:16.3}--\eqref{eq:16.5} that for $t\in[T+1/2,T+1/2+c]$ and $x\in B_c(\0)$ with $c$ sufficiently small,
\begin{align*}
&\!\!\!\!\left|(-\Delta_{\bH^n})^{\frac18}\cos(tD)(f_T)(x)\right|\\
&=|f_T*k_t(x)|\\
&\geq |f_T*[(\sinh R(\cdot))^{-\frac{n-1}2}\tk_t](x)|-|f_T*[(\sinh R(\cdot))^{-\frac{n-1}2}e_t](x)|\\
&\gtrsim \int_{B_{T+\frac12+6c}(\0)\backslash B_{T+\frac12+5c}(\0)}(\sinh d(x,y))^{-\frac{n-1}2}(d(x,y)-t)^{-\frac n2-\frac34}\d y\\
&\qquad -\int_{B_{T+\frac12+7c}(\0)\backslash B_{T+\frac12+4c}(\0)}(\sinh d(x,y))^{-\frac{n-1}2}\left(\O(1)+(d(x,y)-t)^{\frac14-\frac n2}\right)\d y\\
&\sim \e^{\frac{n-1}2T}c^{1-\frac n2-\frac34}\left(1-c\right)\gtrsim \e^{\frac{n-1}2T},
\end{align*}
whence
\begin{align*}
&\!\!\!\!\left(\int_T^{T+1}\left\|(-\Delta_{\bH^n})^{\frac18}\cos(tD)(f_T)\right\|^p_p\d t\right)^{\frac1p}\\
&\geq \left(\int_{T+\frac12}^{T+\frac12+c}\left\|(-\Delta_{\bH^n})^{\frac18}\cos(tD)(f_T)\right\|^p_{L^p(B_c(\0))}\d t\right)^{\frac1p}\gtrsim \e^{\frac{n-1}2T}.
\end{align*}
This combined with \eqref{eq:16.1.1} yields that if \eqref{eq:16.2} holds for all radial functions $f$, then we must have
\[
C_p(T)\gtrsim \e^{\frac{n-1}2T}\e^{-\frac{n-1}pT}=\e^{(n-1)\big(\frac12-\frac1p\big)T}.\qedhere
\]
Hence, the proof of Theorem~\ref{thm:1} is complete.
\end{proof}
\end{proof}
\bigskip

%$$$$$$$$$$$$$$$$$$$$$$$$$$$$$$$$$$$$$$$$$$$$$$$$$$$$$$$$$$$$$$$$$$$$$$$$$$

\section{Endpoint Results of Proposition \ref{prop:3.1} for \texorpdfstring{$n=3$}{n=3}}\label{sec:A}
\setcounter{equation}{0}

In this section, we prove an endpoint result of Proposition \ref{prop:3.1} for the case $n=3$ by using the explicit kernel expression of the wave propagator. More precisely, we have the following result.

\begin{proposition}\label{prop:5}
Suppose $2\leq p<3$ and $f$ is a radial function in $\bH^3$. Then for $T\geq 0$
\begin{align*}
\left(\int_T^{T+1}\|\cos(tD)(f)\|^p_p\d t\right)^{\frac1p}\lesssim_p \e^{2\big(\frac12-\frac1p\big)T}\|f\|_p.\medskip
\end{align*} 
In particular, $u(t,\cdot)\in L^p(\bH^3)$ for almost all times $t$, provided that $f\in L^p(\bH^3)$ is a radial function and $2\leq p<3$.
\end{proposition}

\begin{proof}
Since $f$ is radial, we write $f(x)=h(R(x))$ for some function $h$ defined on $[0,\infty)$ and extend the domain of $h$ from $[0,\infty)$ to $\bR$ by setting $h\equiv 0$ on $(-\infty,0)$. We also set $f_0(r):=(\sinh r)^{\frac2p}h(r)$. In addition, letting $n=3$, \eqref{eq:7} and \eqref{eq:8} become
\begin{align}\label{eq:A1}
\varphi_s(r)=\frac1{\sinh r}\frac{\sin sr}s\quad\text{and}\quad\bc(s)=-\frac\i s,
\end{align}
respectively. In view of \eqref{eq:2-4}, \eqref{eq:10} and \eqref{eq:11},
\begin{align}\label{eq:A12}
\|\cos(tD)(f)\|_p^p=\int_0^\infty\left|\int_0^\infty f_0(r)\,\kappa_t(R,r)\,\d r\right|^p\d R,
\end{align}
where
\begin{align}
\kappa_t(R,r)&=C\left(\frac{\sinh r}{\sinh R}\right)^{1-\frac2p}\int_0^\infty \cos st\,\sin sR\,\sin sr\,\d s\notag\\
&=C\left(\frac{\sinh r}{\sinh R}\right)^{1-\frac2p}\int_{\mathbb R}\big(\e^{\mathrm ist}+\e^{-\mathrm ist}\big)\big(\e^{\mathrm is(r+R)}-\e^{\mathrm is(r-R)}\big)\mathrm d s\notag\\
&=C\left(\frac{\sinh r}{\sinh R}\right)^{1-\frac2p}\big(\delta_0(t+r+R)+\delta_0(r+R-t)-\delta_0(t+r-R)\label{eq:A13}\\
&\hphantom{=C\left(\frac{\sinh r}{\sinh R}\right)^{1-\frac2p}}\ \ \quad -\delta_0(r-R-t)\big)\notag.
\end{align}

Finally, by substituting \eqref{eq:A13} into \eqref{eq:A12}, for $t>0$ we have
\begin{align}
\|\cos(tD)(f)\|_p^p&\lesssim\int_0^\infty(\sinh R)^{2-p}\left[|f_0(t-R)|^p(\sinh(t-R))^{p-2}\right.\label{eq:A14}\\
&\hphantom{\lesssim\int_0^\infty(\sinh R)^{2-p}[}\ \ \left.+|f_0(R-t)|^p(\sinh(R-t))^{p-2}\right.\notag\\
&\hphantom{\lesssim\int_0^\infty(\sinh R)^{2-p}[}\ \ \left.+|f_0(t+R)|^p(\sinh(t+R))^{p-2}\right]\d R.\medskip\notag
\end{align}

In view of \eqref{eq:A14}, we can write
\begin{align*}
\int_T^{T+1}\|\cos(tD)(f)\|^p_p\d t\lesssim \bJ_1+\bJ_2+\bJ_3,
\end{align*}
where
\begin{align*}
\bJ_1&=\int_T^{T+1}\int_0^t |f_0(t-R)|^p(\sinh(t-R))^{p-2}(\sinh R)^{2-p}\d R\,\d t\\
&\lesssim \e^{(p-2)T}\int_0^{T+1}|f_0(r)|^p\left(\int_{\max\{r,T\}}^{T+1}(\sinh(t-r))^{2-p}\d t\right)\d r\\
&\lesssim_p\e^{2p\left(\frac12-\frac1p\right)T}\|f_0\|_{L^p(\bR_+)}^p=C\,\e^{2p\left(\frac12-\frac1p\right)T}\|f\|_p^p
\end{align*}
since $p<3$,
\begin{align*}
\bJ_2&=\int_T^{T+1}\int_t^\infty |f_0(R-t)|^p(\sinh(R-t))^{p-2}(\sinh R)^{2-p}\d R\,\d t\\
&\leq \int_T^{T+1}\int_0^\infty |f_0(r)|^p(\sinh(r+t))^{p-2}(\sinh(r+t))^{2-p}\d r\,\d t\\
&=\|f_0\|_{L^p(\bR_+)}^p\lesssim \|f\|_p^p
\end{align*}
since $p\geq 2$, and
\begin{align*}
\bJ_3&=\int_T^{T+1}\int_0^\infty |f_0(t+R)|^p(\sinh(t+R))^{p-2}(\sinh R)^{2-p}\d R\,\d t\\
&\leq \e^{(p-2)T}\int_T^{T+1}\int_0^\infty |f_0(R+t)|^p\left(\frac{\e^R}{\sinh R}\right)^{p-2}\d R\,\d t\\
&\lesssim \e^{(p-2)T}\left(\int_0^1R^{2-p}\left(\int_{T+R}^{T+R+1}|f_0(r)|^p\d r\right)\d R+\int_T^{T+1}\int_t^\infty|f_0(r)|^p\d r\,\d t\right)\\
&\lesssim_p\e^{2p\left(\frac12-\frac1p\right)T}\|f_0\|_{L^p(\bR_+)}^p\lesssim\e^{2p\left(\frac12-\frac1p\right)T}\|f\|_p^p
\end{align*}
since $p<3$. This yields the desired estimate. The proof of Proposition~\ref{prop:5} is complete.
\end{proof}

% \bibliographystyle{amsplain-modified}
% \bibliography{references}

\begin{remark} (i)
It is not known for endpoint results of Proposition \ref{prop:3.1} for the cases $n=2$ or $n>3$.

(ii) In light of Theorem \ref{thm:1}, it would be interesting to conjecture the following statement: Let $u=u(t,x)$ be a solution to wave equation \eqref{eq:1}. Suppose $2n/(n-1)\leq p<\infty$ and $f$ is a function in $\bH^n$. Let $\gamma<1/p$. Then for $T\geq 0$,
\begin{align}\label{eq:4.3}
\left(\int_T^{T+1}\|u(t,\cdot)\|^p_p\d t\right)^{\frac1p}\lesssim_{p,\gamma} \e^{(n-1)\big(\frac12-\frac1p\big)T}\left\|(-\Delta_{\bH^n})^{\frac{n-1}2\left(\frac12-\frac1p\right)-\frac\gamma2}f\right\|_p.
\end{align}

For some range of $p$ and $\gamma$, estimate \eqref{eq:4.3} with the exponential part $\e^{(n-1)(1/2-1/p)T}$ of the bound replaced by $\e^{(n-1)(3/2-1/p)T}$ can be derived using the local smoothing estimates for Fourier integral operators (see e.g. \cite{BHS, GHLW26, GLMX23}) together with the localization argument of Ionescu \cite{Ion00}. However, whether the threshold of $p$ can attain $2n/(n-1)$ is open for $n\geq 2$, simultaneously the bound on $T$ there is considerably worse. For details see the proof of \cite[Lemma 4.9]{CSWY25}.
\end{remark}
\medskip

\noindent
{\bf Acknowledgments}: The authors were supported by National Key R$\&$D Program of China 2022YFA1005700. P. Chen was supported by NNSF of China 12171489 and 12671127. X.T. Duong was supported by the research grant ARC DP260101083 from the Australian Research Council. L. Yan was supported by NNSF of China 12571111.

\providecommand{\bysame}{\leavevmode\hbox to3em{\hrulefill}\thinspace}
\providecommand{\href}[2]{#2}

\end{document}